\documentclass{article}
\usepackage[T2A]{fontenc}
\usepackage[cp1251]{inputenc} 
\usepackage[english]{babel}
\usepackage{amsfonts,amssymb,mathrsfs,amscd}
\usepackage{amsthm,graphicx}

\newtheorem*{theorem}{Theorem}

\theoremstyle{definition}

\def\R{{\mathbb R}}

\def\PP{{\mathbb P}}

\title{Braids act on configurations of lines}

\author{Vassily Olegovich Manturov \footnote{Moscow Institute of Physics and Technology}}

\begin{document}

\maketitle

To my mother Elena Ivanovna Manturova, with love.

\begin{abstract}
Similar pictures appear in various branches of mathematics. Sometimes
this similarity gives rise to deep theorems.

Mentioning such a similarity between hexagonal tilings, cubes  in
3-space, configurations of lines and braid groups, we
prove that braids act on configurations of lines.

\end{abstract}



\section{Introduction and the main result}
Braids correspond to dynamics of points moving on the plane.

Let us consider braids on $\R{}P^{2}$ and, by using duality, we shall move
$n$ pairwise distinct lines. Generically, these lines intersect in $n\choose 2$ points.
These lines split the plane into regions.

By Euler characteristic reasons, these lines split the projective plane into
${n \choose 2}+1$ regions.

If $n$ is even then the regions can be coloured in two colours in a checkerboard fashion.

From now on, we assume $n$ to be even.

We get a bicoloured graph, see Fig.\ref{below}.

\begin{figure}
\centering\includegraphics[width=150pt]{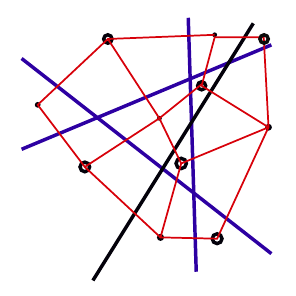}
\caption{Line configuration and quadrilateral tiling}
\label{below}
\end{figure}

As lines move generically, the dual graph undergoes a hexagonal flip,
see \footnote{Some pictures are kindly borrowed from the paper
\cite{FominPylavskyy}} Fig. \ref{Desarguesflip}.

With a tiling of a $2n$-gon, associate a configuration of points and lines.
With a hexagonal flip associate a Desargues flip. In Fig.\ref{Desargues}, one can see the illustration of the famous Desargues theorem.

\begin{figure}
\centering\includegraphics[width=150pt]{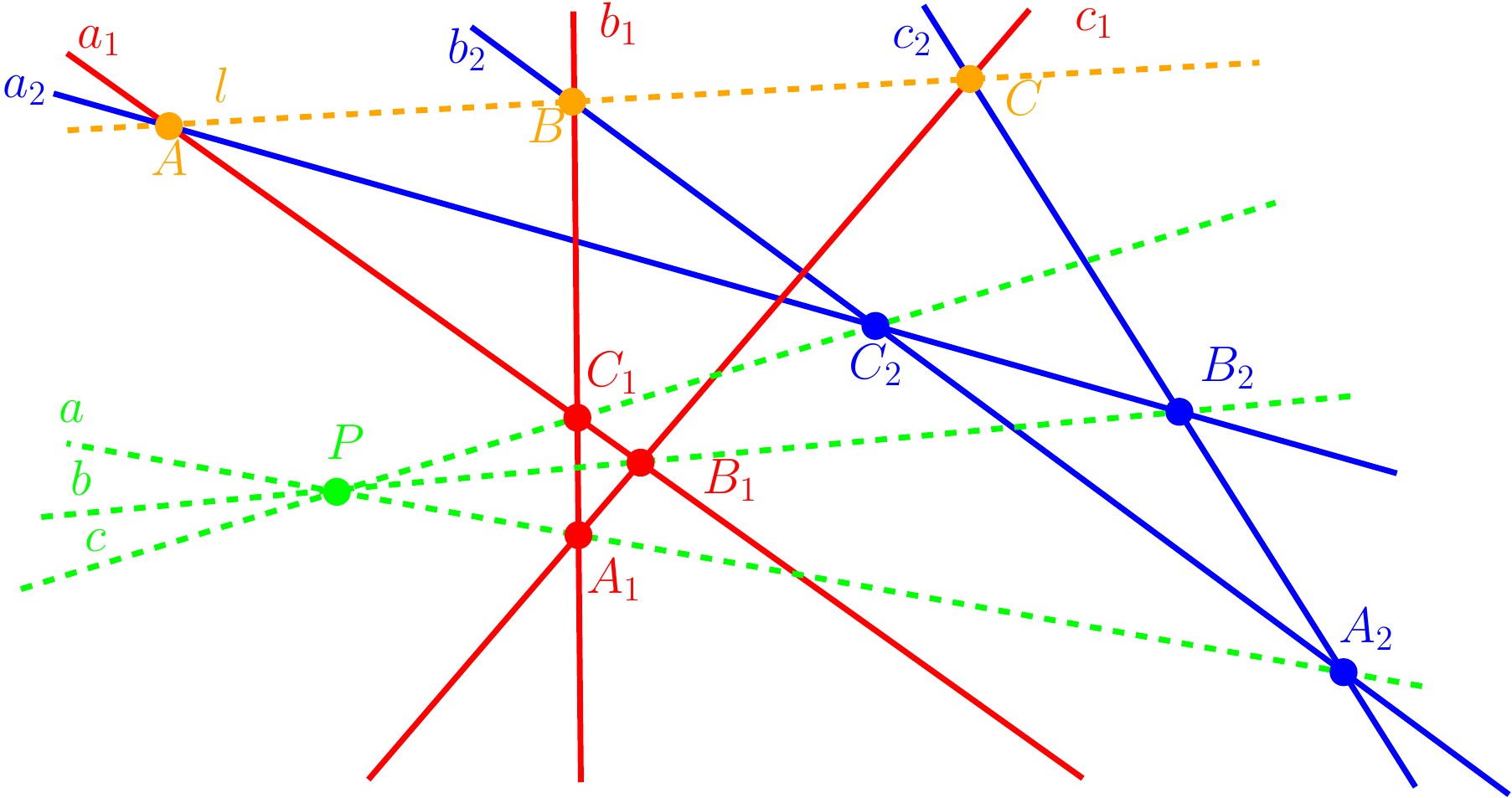}
\caption{The Desargues configuration of lines}
\label{Desargues}
\end{figure}

Passing from a point where three concurrent lines meet to the corresponding line can
be encoded in terms of diagrams of dots. In Fig. \ref{Desarguesflip} lines and points of
the configuration are depicted by black and white dots, and the fact that a point is incident
to a line is reflected by an edge connecting one dot to another.

{\bf
This means that starting from a generic configuration of lines and points whose
``dotted picture'' contains a hexagon, we replace three concurrent lines with three points
on the same line or vice versa. So, with each braid we associate a sequence of
Desargues flips.
}

\begin{figure}
\centering\includegraphics[width=150pt]{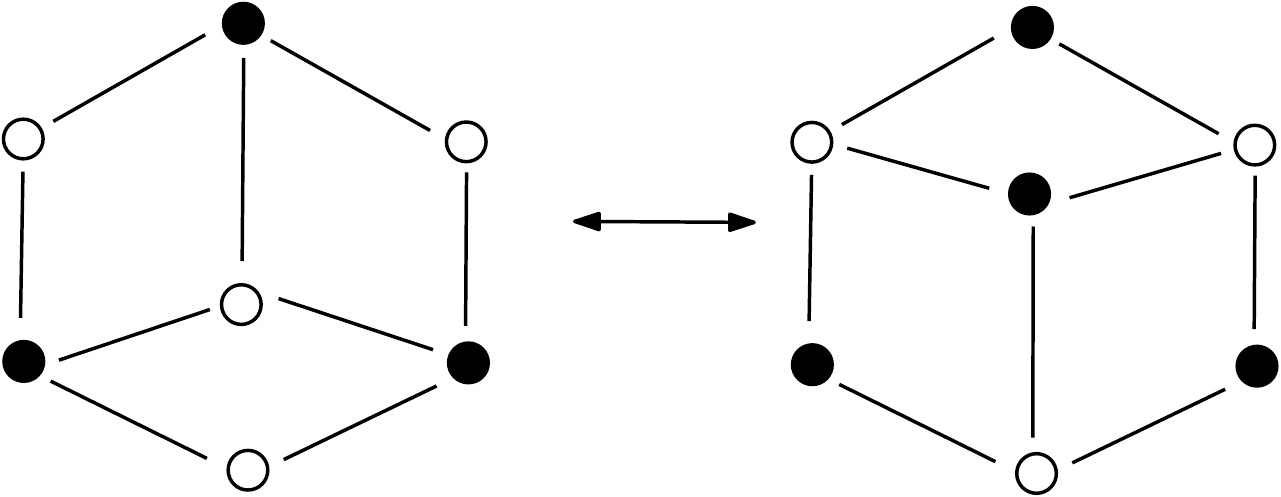}
\caption{The Desargues flip}
\label{Desarguesflip}
\end{figure}


Now the new heroes come into play: the notions of {\bf tile} and {\bf coherent tile}.

Here we just cite Fomin and Pylavskyy \cite{FominPylavskyy} verbatim and cite their paper in quotes.

``

Let $\PP$ be a real or complex finite-dimensional projective space.
(In this paper, we focus on applications where $\dim\PP=2$ or $\dim\PP=3$, i.e., $\PP$~is a plane or a 3-space.)

We denote by $\PP^*$ the set of hyperplanes in~$\PP$.
In particular, when $\PP$ is a plane, the elements of $\PP^*$ are lines.
A~point $A\in \PP$ and a hyperplane $\ell\in\PP^*$ are called \emph{incident} to each other if $A\in\ell$.

We denote by $(AB)$ (resp., $(ABC)$) the line passing through two distinct points $A$ and~$B$
(resp., the plane passing through distinct points $A, B, C$).

Throughout this paper, a  \emph{tile} is a topological quadrilateral
(that is, a closed oriented disk with four marked points on its boundary)
whose vertices are clockwise labeled $A, \ell, B, m$,
where $A,B\in\PP$ are points and $\ell,m\in\PP^*$ are hyperplanes:
\begin{equation}
\label{eq:AlBm}
\centering\includegraphics[width=60pt]{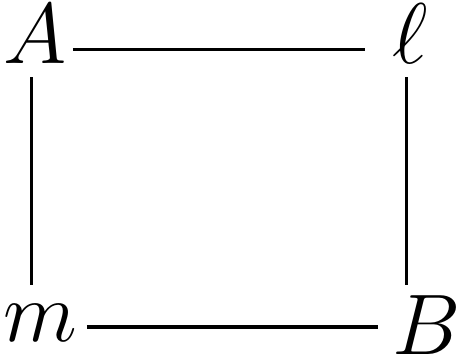}
\end{equation}
Such a tile is called \emph{coherent} if
\begin{itemize}
\item neither $A$ nor $B$ is incident to either $\ell$ or~$m$;
\item either $A=B$ or $\ell=m$ or
else the line $(AB)$ and the codimension~$2$ subspace $\ell\cap m$ have a nonempty intersection.
\end{itemize}

In the case of the projective plane ($\dim\PP=2$), a coherent tile involves
two points $A$, $B$ and two lines $\ell$, $m$ not incident to them such that either $A=B$ or $\ell =m$ or else the line $(AB)$ passes through the point $\ell \cap m$. See Figure~\ref{fig:coherent-tile}.
\begin{figure}
\centering\includegraphics[width=120pt]{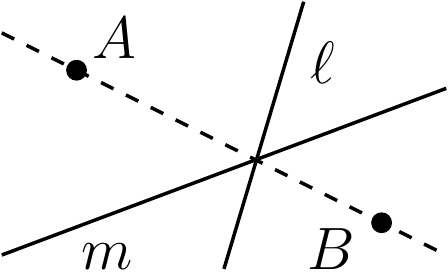}
\caption{Definition of a coherent tile} \label{fig:coherent-tile}
 \end{figure} 

''

Consider a pure $n$-strand braid $\beta$ in $\R{}P^{2}$ as a closed path  in the configuration of pairwise distinct (projective) lines
in $\R{}P^{2}$. We may assume that the initial configuration (collection of lines) is generic
in the sense that no three lines pass through the same point.

Let us fix such a configuration $(l_{1},\cdots, l_{n})$
(no matter which one) and the corresponding cell decomposition of $\R{}P^{2}$ in
the checkerboard fashion.
The dual graph is a quadrilateral splitting of $\R{}P^{2}$, which is bicolored
black and white.

We say that a collection of points and hyperplanes in $\PP$ such that: with each black dot we associate a point in $\PP$,
with each white dot we associate a hyperplane in $\PP$ and
require that the four vertices corresponding to any quadrilateral
form a coherent tile.

We call such an object a {\em coherent configuration of points
and hyperplanes respecting $(l_{1},\dots, l_{n})$.}

A Desargues flip changes one coherent configuration to another.
A (pure) braid in $\R{}P^{2}$ gives rise to a sequence of Desargues flip.

Hence, it gives rise to an action on such collections of flips.

\begin{theorem}
 Isotopic braids give rise to equal actions of braids on configurations of points and hyperplanes in $\PP$.
\end{theorem}

\begin{proof} One can consider a generic isotopy of braids. Standardly,
relations correspond to codimension 2 events (generators correspond to
codimensional one events, which are the flips given above). This gives rise to three types
of transformations.

One of them corresponds to the situation when we perform one
Desargues flip ``back and forth'' (see Fig. \ref{Desarguesflip}).

The other one corresponds to the situation which can be characterised as
``independent events commute'': two flips occur in two non-overlapping hexagons.

Now, we take the statement [Theorem 9.13,page 18] from the paper by Fomin and Pylavskyy.
Desargues flips satisfy the octogon relation (see Fig.\ref{octogon1}).
\begin{figure}
\centering\includegraphics[width=170pt]{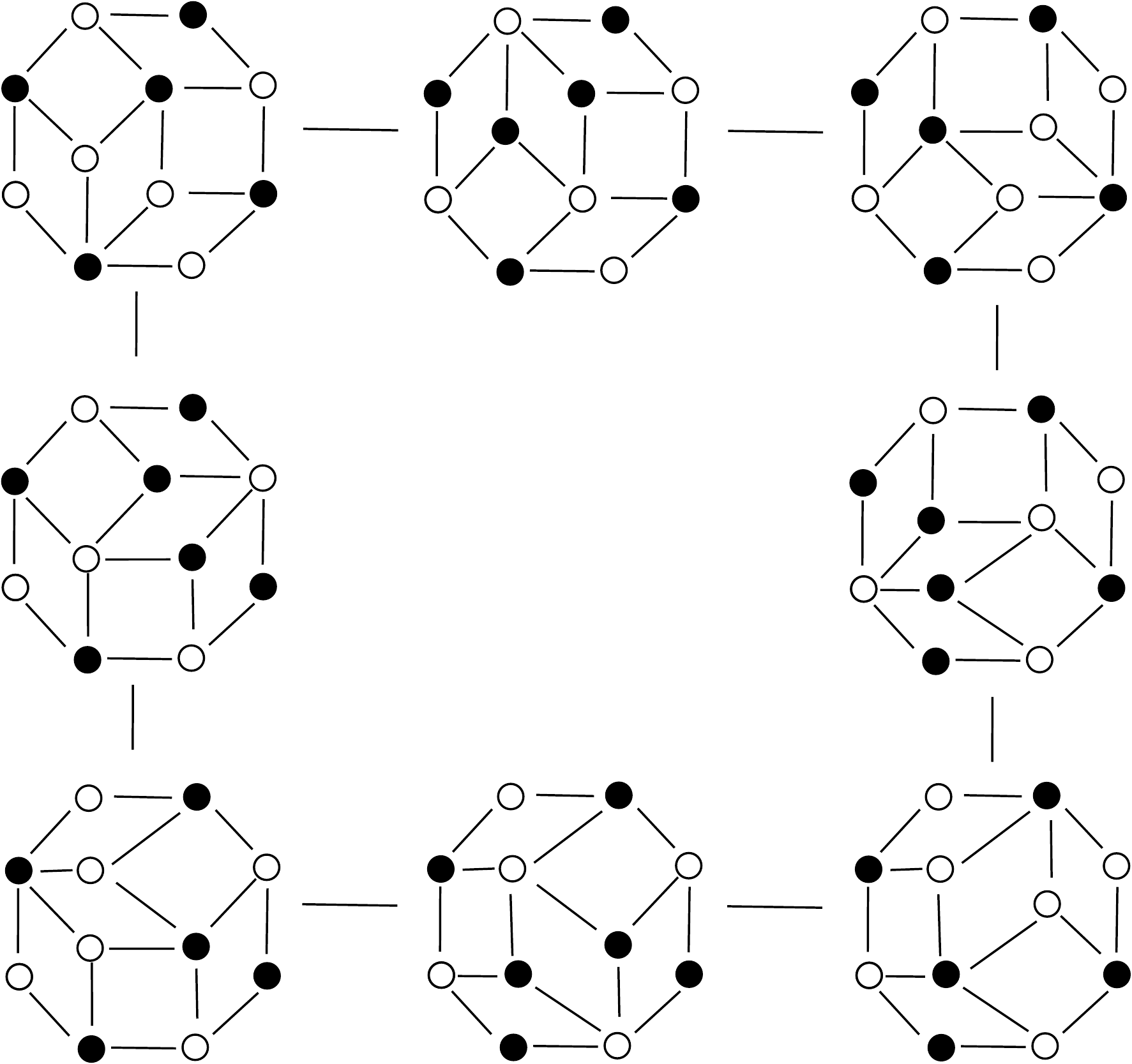}
\caption{The octogon relation}
\label{octogon1}
\end{figure}

In other way of looking at the octogon relation is shown in Fig. \ref{octogon2}.
\begin{figure}
\centering\includegraphics[width=150pt]{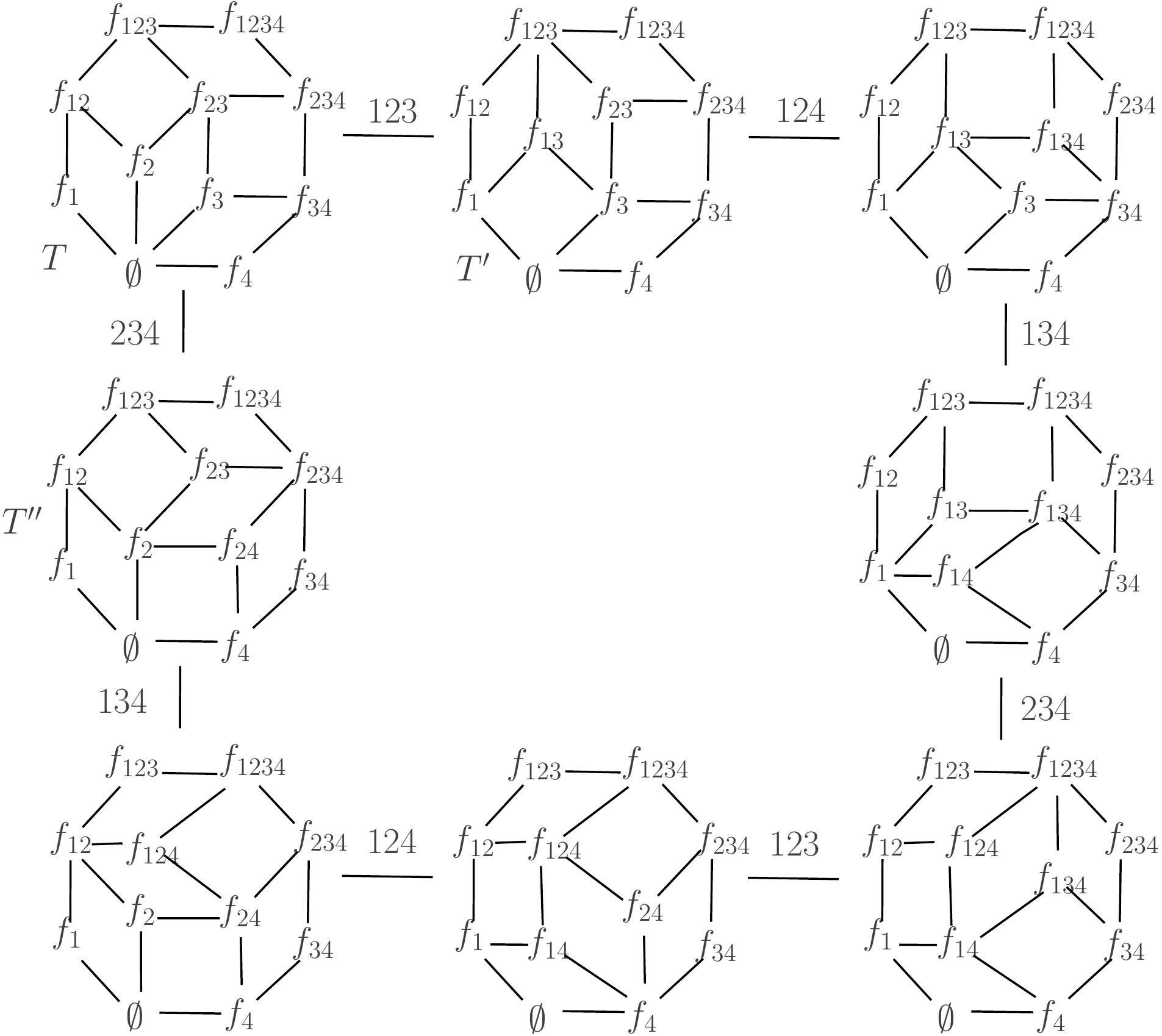}
\caption{The octogon relation from $G_{n}^{3}$ prospective}
\label{octogon2}
\end{figure}

\end{proof}






\section{Just a couple of examples}

In Fig. \ref{4lines} below. we show how four pairwise distinct
lines in $\R{}P^{2}$ can move and pass through triple points.

\begin{figure}

\centering\includegraphics[width=200pt]{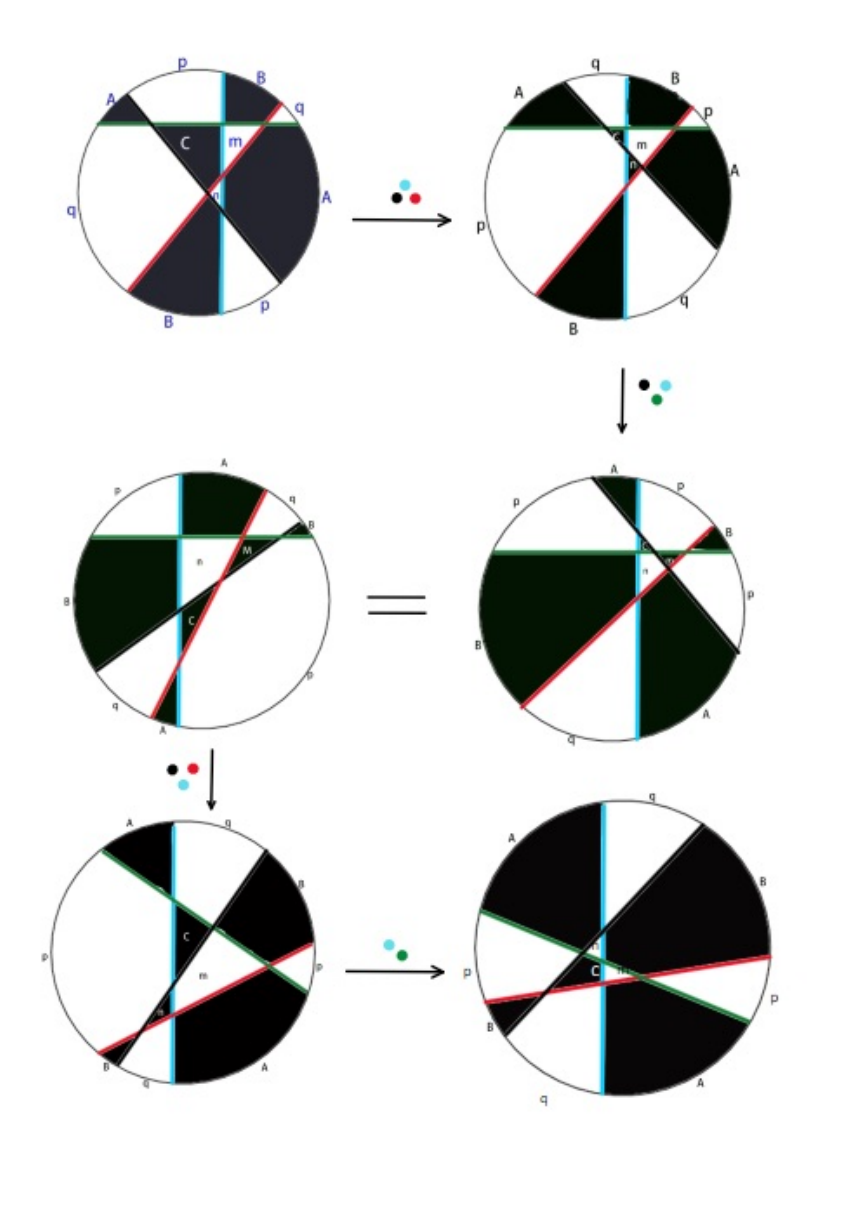}

\caption{Four moves for four lines on $\R{}P^{2}$}
\label{4lines}
\end{figure}

With each such figure we associate a collection of black and
white points corresponding to black and white cells.

Taking the data underlying the tiles corresponding to the
upper left picture, see Fig.\ref{flipflip}, we can transform it four times and get to the
data in the upper right corner.

\begin{figure}

\centering\includegraphics[width=150pt]{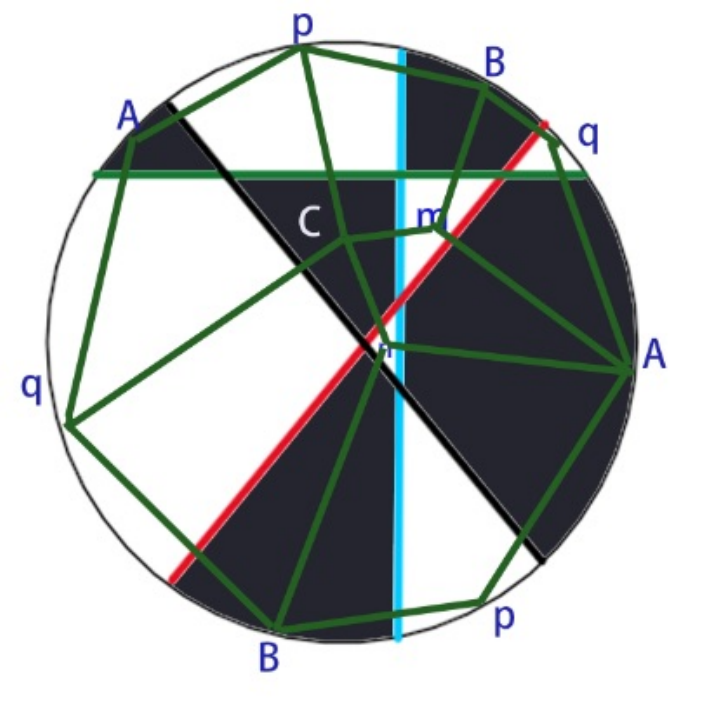}

\caption{Four moves for four lines on $\R{}P^{2}$}
\label{flipflip}
\end{figure}

\section{Further directions: Invariants of higher dimensional manifolds}
In \cite{ManturovNikonov} it was shown that
{\em the same picture (say, pentagon) appears in different areas of mathematics.}

This picture gives rise to a pentagon relation (equation).

In \cite{ManturovWan} it is shown that one can solve equations by looking at (maybe, some other)
other pictures.

Taking the origin of the initial picture (braid groups), we see that the initial
object of study (in our case, braids) acts on other objects which appear under different
names in different areas of mathematics.
This led to a concrete result (braid group action on labelled triangulations).

In \cite{FominPylavskyy} it was demonstrated that theorems about lines

Lines and hyperplanes. Higher analogues of braids (related to $G_{n}^{k},k>4$,
should act on higher analogues of line arrangements.

The further plan is as follows.
Main relation in the $G_{n}^{k+1}$-group: tiling of $\R^{k}$ by hyperplanes.
(2k)-term relation.
Action of higher braids on flags etc.

\section{What should be the ``correct'' form of theorem(s) from the present paper?}

First of all, by no means we pretend to have any ``general picture''. Even the
survey-like paper \cite{FominPylavskyy} did not cover ``all'' theorems about point and line configurations.

We mention some similarity between the principal $G_{n}^{3}$-relation needed for
constructing an invariant of braids, and the octogon relation for configurations
of lines and planes.

From Bourbaki's point of view, it is not even an action: input objects and output
objects are of different sorts: the number of lines and planes changes after the Desargues
flip.

Formally, one can construct a homomorphism from the pure braid group to
performing lots of accurate calculations similar to \cite{MNGn3}.

\section{Acknowledgements}
I am very grateful to Igor Mikhailovich Nikonov, Louis Hirsch Kauffman and Seongjeong Kim for permanent
discussion of my current work.

I am extremely grateful to Seongjeong Kim and Zichang Han for their help in preparing the text.

I am indebted to Clifford Henry Taubes for pointing to the paper by S.V.Fomin and
P.Pylavskyy.

\end{document}